\documentclass[final,onefignum,onetabnum]{siamart250211}
\usepackage{mathptmx}
\usepackage{setspace}
\usepackage[numbers]{natbib}
\usepackage{algorithm}
\usepackage{stmaryrd}
\usepackage{mathtools}
\usepackage{amsmath} 
\usepackage{amssymb}
\usepackage{amsfonts}
\usepackage{mathrsfs}
\usepackage{enumitem}

\usepackage{diagbox}
\usepackage{capt-of}
\usepackage{bm}
\usepackage{xcolor}

\usepackage[normalem]{ulem}
\usepackage{oplotsymbl}
\usepackage{multirow}
\usepackage{algpseudocode}
\usepackage{tabto}
\usepackage[usestackEOL]{stackengine}[2013-10-15]
\usepackage{subcaption} 
\definecolor{orange1}{HTML}{FDAE61}
\definecolor{green1}{HTML}{66C2A5}

\newsiamremark{remark}{Remark}
\usepackage{tikz}
\usetikzlibrary{patterns,positioning}

\DeclareMathOperator{\spn}{span}

\def\x{\bm{x}}
\def\l{\bm{\ell}}
\def\E{\bm{E}}
\def\j{\bm{j}}
\def\O{\mathcal{O}}

\usepackage{tikz}
\usetikzlibrary{patterns,positioning}

\newcommand{\meshdots}[5]{%
  \begin{scope}[shift={(#1,#2)}]
    \pgfmathtruncatemacro{\ma}{2^(#3)}
    \pgfmathtruncatemacro{\mb}{2^(#4)}
    \foreach \p in {1,...,\ma}{
      \foreach \q in {1,...,\mb}{
        \fill[black] ({#5*(\p-0.5)/\ma},{#5*(\q-0.5)/\mb}) circle (0.027);
      }
    }
  \end{scope}
}

\newcommand{\fullmeshdots}[5]{%
  \begin{scope}[shift={(#1,#2)}]
    \pgfmathtruncatemacro{\ma}{2^(#3)}
    \pgfmathtruncatemacro{\mb}{2^(#4)}
    \foreach \p in {1,...,\ma}{
      \foreach \q in {1,...,\mb}{
        \fill[black] ({#5*(\p-0.5)/\ma},{#5*(\q-0.5)/\mb}) circle (0.020);
      }
    }
  \end{scope}
}

\newcommand{\CONCACE}{{Inria, Concace joint team between Airbus CR\&T, Cerfacs and Inria}, {Talence}, {France}, \email{clement.guillet@inria.fr}}
\headers{}{C. Guillet}
\title{Exact hierarchical algorithms for accelerating particle--mesh coupling in sparse-grid particle-in-cell methods}

\author{Cl\'ement Guillet\thanks{\CONCACE} }

\begin{document}
\maketitle

\begin{abstract}		
In this paper, we propose two hierarchical algorithms for charge deposition and electric-field interpolation that apply to both the sparse-grid combination technique (SGCT-PIC) and hierarchical sparse-grid (HSG-PIC) particle-in-cell methods. The two algorithms are inspired by the fast multipole method (FMM) and exploit clusters of particles associated with a directed acyclic graph (DAG) of particle-populated boxes to reduce the number of particle--mesh interactions. The particle--mesh interactions are governed by piecewise-polynomial kernels, so that the associated multipole expansions are exact, requiring neither truncation nor approximation, and are valid in both near- and far-field regions, thereby eliminating the need for multipole-to-local translations. The arithmetic complexity of the charge deposition and field interpolation steps is reduced from $\O(p^d n^{d-1}N)$ to $\O(p^d(N+M))$, where $M=2^{dn}$ denotes the number of full-grid mesh nodes and is typically no larger than the particle population in the considered regime, $M\lesssim N$. Numerical experiments in two-dimensional configurations demonstrate charge-deposition speedups of $8.2\times$--$66.9\times$ for SGCT-PIC and $3.1\times$--$18.8\times$ for HSG-PIC, and field-interpolation speedups of $4.1\times$--$62.6\times$ and $4.2\times$--$13.7\times$, respectively, depending on the particle-per-cell ratio, while preserving the exact particle--mesh interactions. The speedups increase with the particle-per-cell ratio, reflecting the reduced dependence of the hierarchical algorithms on the number of particles and their increasing advantage for large particle populations.
\end{abstract}



\section{Introduction}
The simulation of kinetic plasmas requires the computation of the collective dynamics of a large number of charged particles interacting through long-range electromagnetic forces.  Direct particle-to-particle evaluation of these interactions has a computational complexity of $\O(N^2)$ for $N$ particles, making it prohibitively expensive for large-scale simulations. Particle-in-cell (PIC) methods~\cite{hockney88,birdsall18} overcome this difficulty by introducing a spatial mesh on which the field equations are solved, thereby avoiding the direct evaluation of all particle-to-particle interactions. For an FFT-based field solver, the resulting computational complexity is typically $\O(p^d(N+M\log M))$, where $M=2^{dn}$ denotes the number of mesh nodes, with $n$ the mesh resolution, $p$ the degree of the interpolation functions, and $d$ the spatial dimension. In particular, the number of mesh nodes is typically much smaller than the number of particles, $M\ll N$. More specifically, PIC methods couple a Lagrangian discretization of the Vlasov equation, based on the integration of particle trajectories, with a mesh-based discretization of Poisson's equation for the computation of the self-consistent electrostatic field. 

The particle--mesh coupling, however, comes at a fundamental cost: sampling the distribution function
with a finite number of particles inherently introduces statistical noise into mesh-based quantities. This statistical error decreases slowly, at a rate proportional to the inverse square root of the mean number
of particles per cell, so that, in many practical simulations, achieving a reasonable statistical error requires a prohibitively large number of particles.

Sparse-grid methods, originally introduced in~\cite{griebel90,bungartz91} to mitigate the curse of dimensionality in the numerical solution of partial
differential equations (PDEs), have been incorporated into PIC algorithms to reduce statistical noise~\cite{ricketson17}. Two different frameworks have been
proposed. In the sparse-grid combination technique (SGCT) PIC method
~\cite{ricketson17,deluzet22,guillet24}, mesh-based quantities are reconstructed by linearly combining solutions computed on a hierarchy of coarse component
grids with appropriately chosen combination coefficients. The method has been assessed on classical kinetic plasma benchmarks~\cite{muralikrishnan21,deluzet22} and applied to low-temperature, including collisional, plasma discharges and drift instabilities in Hall thrusters~\cite{garrigues21,garrigues21-1,garrigues24,garrigues24-1}. In the hierarchical sparse-grid (HSG) PIC method
~\cite{guillet25,deluzet26,deluzet26-1}, mesh-based quantities are instead computed variationally using a Galerkin method on an approximation space
constructed from a truncated tensor-product decomposition of one-dimensional multiresolution spaces~\cite{bungartz91,bungartz04}. In both approaches, the sparse-grid mesh provides a larger average number of particles per cell than a full fine grid, thereby allowing the total number of particles to be substantially reduced while maintaining a comparable level of statistical noise.

However, the particle--mesh coupling remains a major computational bottleneck, primarily because the sparse-grid mesh consists of $\O(n^{d-1})$ component grids or hierarchical subspaces, depending on the formulation. Although the compact support of the basis functions limits the number of interactions per particle, these interactions must be computed
on each component grid or for each subspace. Consequently, the cost of the particle--mesh operations, namely charge deposition and electric-field interpolation, scales as $\mathcal{O}(p^d n^{d-1}N )$. A first strategy to reduce the computational cost of electric-field interpolation, called hierarchization, was proposed in~\cite{deluzet22-1}. The field contributions computed on the component grids are first combined onto the full grid using unidirectional principles, after which the reconstructed field is interpolated from the full grid to the particles. This reduces the complexity of the interpolation to $\O(p^d(N+M))$. However, the intermediate recombination of the component-grid fields on the full grid introduces an interpolation error. To reduce the computational cost of charge deposition, parallelization strategies have been developed, exploiting cache efficiency and multiple levels of parallelism on shared-memory CPUs~\cite{deluzet22-1} and GPUs~\cite{deluzet23}. Reported speedups exceed $100\times$ on 128 CPU cores and on a single NVIDIA Tesla V100 GPU in three dimensions.

In this work, we follow a different strategy to reduce the arithmetic complexity of the particle--mesh coupling, drawing inspiration from hierarchical algorithms developed for $N$-body problems~\cite{hernquist88}, such as tree-based methods~\cite{barnes86,appel85} and, more specifically, the fast multipole method (FMM)~\cite{greengard87,ambrosiano88,carrier88}. The FMM exploits a hierarchical organization of particles into clusters to avoid explicitly evaluating all pairwise interactions, which otherwise requires $\O(N^2)$ operations for $N$ particles. Groups of particle-to-particle interactions are instead approximated by interactions between multipole and local representations in regions where the kernel is sufficiently smooth, resulting in a computational complexity that scales as $\O(N)$ or $\O(N\log N)$ while controlling the approximation error. Inspired by this hierarchical principle, we replace the direct evaluation of all particle--mesh-node interactions in the charge-deposition and electric-field operations by interactions between clusters of particles and mesh nodes, reducing the arithmetic complexity from $\O(p^d n^{d-1}N)$ to $\O(p^d(N+M))$. In the present setting, these interactions are represented by exact multipole and local polynomial expansions, rather than by truncated approximations, so that no additional error is introduced.

Although the proposed method shares the overall hierarchical structure of the FMM, it is not a direct application of the method to sparse-grid particle--mesh interactions and differs from the FMM in several fundamental respects. First, in the present setting, particle--mesh interactions are governed by piecewise-polynomial kernels~\cite{deluzet25,guillet26}. Consequently, the associated multipole expansions are exact and require neither truncation nor approximation. Moreover, these expansions are globally valid, rather than being restricted to a far-field region as in the FMM. A single expansion can therefore represent the contribution of a particle cluster to all target mesh nodes, without requiring separate multipole-to-local translations.  Second, the underlying hierarchical structure is a directed acyclic graph (DAG) of particle-populated boxes at different resolutions, in which a child box may have multiple parents. This DAG follows the same multilevel dyadic partition of the domain as the sparse-grid mesh, which provides a simple and inexpensive admissibility criterion for identifying interactions that vanish identically.

We propose two hierarchical, FMM-inspired algorithms for the charge deposition and electric-field interpolation steps that apply to both the SGCT- and HSG-PIC methods. Their efficiency is assessed through numerical experiments in two-dimensional configurations. Depending on the particle-per-cell ratio, charge-deposition speedups of $8.2\times$--$66.9\times$ for SGCT-PIC and $3.1\times$--$18.8\times$ for HSG-PIC, and field-interpolation speedups of $4.1\times$--$62.6\times$ and $4.2\times$--$13.7\times$, respectively, are reported. These speedups translate into overall speedups of $7\times$--$45\times$ for SGCT-PIC and $3\times$--$11\times$ for HSG-PIC over the total PIC cycle. In all cases, the proposed and standard implementations produce identical results up to roundoff errors. 

The remainder of the report is organized as follows. \Cref{sec:prel} introduces the PIC model and particle--mesh operations for the SGCT- and HSG-PIC methods. \Cref{sec:method} presents the proposed hierarchical algorithms and establishes their arithmetic complexity. \Cref{sec:num} reports numerical comparisons with the standard implementations. Finally, \cref{sec:concl} summarizes the conclusions and discusses perspectives for future work.

 \section{Preliminaries}
 \label{sec:prel}
 \subsection{Continuous model}
 Let $d$ denote the dimensionality of the problem, and $\Omega\subset \mathbb{R}^d$ be the spatial domain of interest, which, in this report, is the $n$-dimensional torus $\Omega=\mathbb{T}^d$. 

The ions are treated as a neutralizing background and all variables are expressed in dimensionless form, with characteristic scales being the Debye length and the plasma period, given by
\[
\lambda_D = \sqrt{\frac{\varepsilon_0 T_e}{q_e n_0}}, \qquad 
\omega_p^{-1} =\left(\sqrt{\frac{q_e n_0}{m_e \varepsilon_0}}\right)^{-1}.
\]
In this normalization setting, the electron mass $m_e$, characteristic temperature $T_e$, charge $q_e$,  typical electron density $n_0$ and vacuum permittivity $\varepsilon_0$ are set to unity.

The noncollisional, nonrelativistic, dimensionless Vlasov--Poisson system with an external magnetic field $\bm{B}: \Omega\to \mathbb{R}^d$ is considered:
\begin{equation*}
    \left\{\begin{aligned}
   & \frac{\partial f}{\partial t}
+ \bm{v} \cdot \bm{\nabla}_{\bm{x}} f
- \left( \bm{E} + \bm{v} \times \bm{B} \right) \cdot \bm{\nabla}_{\bm{v}} f = 0, \\ 
&\nabla \cdot \bm{E}=1-\rho, \quad \bm{E} = -\bm{\nabla}\Phi,
\end{aligned}
\right.
\end{equation*}
where the electron charge density is defined by
\[
\rho = \int_{\mathbb{R}^d} fd \bm{v}.
\]
In the above, $f$ is the phase-space distribution function of the electron species, $\bm{E}$ and $\Phi$ are the electric field and potential. The notation $\bm{\nabla}_*$ denotes the gradient operator with respect to the variable $*$, and we use the shorthand notation $\bm{\nabla} := \bm{\nabla}_{\bm{x}}$. 

 \subsection{PIC approximations and sparse-grid PIC methods}
 Particle-in-cell (PIC) approximations are reference particle methods applied to plasma physics. Recently, sparse-grid approximations have been incorporated into PIC algorithms to reduce their computational cost. Two distinct approaches have been proposed, leading to the sparse-grid combination technique PIC (SGCT-PIC)~\cite{ricketson17,deluzet22} and the hierarchical sparse-grid PIC (HSG-PIC)~\cite{deluzet26} methods. 
 
 These two methods preserve the overall structure of standard PIC algorithms, consisting of charge deposition, field solve, field interpolation, and particle push, while differing in the implementation of these individual steps. We first describe the common ingredients shared by standard and sparse-grid PIC approximations and then detail the implementation of the particle--mesh coupling steps.

\subsubsection{Common ingredients of PIC approximations}
\paragraph{Particle approximation of the distribution function}
The distribution function is approximated by a collection of $N$ macro-particles sampling the phase space. Denoting by $\delta$ the Dirac delta distribution, the particle approximation of the distribution function is defined by
\begin{align*}
f_N(\bm{x},\bm{v},t)
=
\sum_{s=1}^N
w_s
\delta\big(\bm{x}-\bm{x}_s(t)\big)
\delta\big(\bm{v}-\bm{v}_s(t)\big),
\end{align*}
where $w_s$ is the weight of the $s$th particle, representing the volume of phase space carried by that particle. We assume uniform particle weights given by
\begin{align*}
w_s
=
\frac{1}{N}
\int_{\Omega} \rho(\bm{x},0)\,\mathrm{d}\bm{x}.
\end{align*}
The position and velocity of the $s$th particle, denoted by $(\bm{x}_s(t),\bm{v}_s(t))$ for $s=1,\ldots,N$, are time-dependent quantities.

\paragraph{Evolution of particles}
The particles evolve according to Newton's equations of motion:
\begin{align*}
\frac{d  \bm{x}_s(t)}{d t} = \bm{v}_s(t), \qquad \frac{d  \bm{v}_s(t)}{d t} = \Big(\bm{v}_s(t)\times \bm{B}\big(\bm{x}_s(t)\big) +\bm{E}\big(\bm{x}_s(t),t\big)\Big).
\end{align*}
These equations are typically discretized in time using explicit schemes such as the leapfrog method or Runge--Kutta integrators. At each time step, denoted by $t^\kappa$ for $\kappa=0,\ldots,T/\Delta t$ where $\Delta t $ is the time step size, the electric field is computed on a spatial mesh from the charge density, which must first be approximated from the particle distribution.

\paragraph{Spatial mesh}
Let $n\in\mathbb{N}_0$ denote the spatial mesh resolution, and let the associated mesh size be $h=2^{-n}$. The standard PIC method is based on a uniform Cartesian mesh with mesh size $h$ in each spatial direction.

More generally, let $\bm{\ell}=(\ell_1,\ldots,\ell_d)\in\{0,\ldots,n\}^d$ be a multi-index, and define the associated anisotropic mesh size by
\[
h_{\bm{\ell}}:= \big(2^{-\ell_1},\ldots,2^{-\ell_d}\big).
\]
We introduce the anisotropic Cartesian grid
\begin{align}
\label{eq:omega_h}
\Omega_{h_{\bm{\ell}}}:= \left\{\bm{x}_{\bm{\ell},\bm{j}}:=\bm{j}h_{\bm{\ell}}\;\middle|\;\bm{j}\in I_{h_{\bm{\ell}}}\right\}\subset\Omega,
\quad
I_{h_{\bm{\ell}}}:=\llbracket 0,h_{\ell_1}^{-1}-1\rrbracket\times\cdots\times \llbracket 0,h_{\ell_d}^{-1}-1\rrbracket \subset\mathbb{N}^d,
\end{align}
which serves as a building block for the meshes used in sparse-grid PIC methods.


\paragraph{Charge density deposition onto the mesh}
To derive an estimator for the particle density, we integrate the distribution function over the velocity space and obtain the raw Monte Carlo density estimator:
\begin{align*}
\rho_{N}(\bm{x},t) = \sum_{s=1}^N w_s \delta\big(\bm{x}-\bm{x}_s(t)\big)\,.
\end{align*}
This estimator cannot be directly used on a discrete mesh due to its singular nature. To obtain a smoothed density function on the grid, a regularization operation is required. Different regularization operators, which will be described in detail later, are considered in the standard-, SGCT-, and HSG-PIC methods. After regularization, an approximation of the Monte Carlo density estimator is obtained on the mesh.

\paragraph{Electric field computation}
The regularized Monte Carlo density estimator serves as the right-hand side of a Poisson equation, which is solved on the mesh using a mesh-based method such as
finite difference, fast Fourier transform (FFT) or a Galerkin method. The electric field is subsequently computed on the grid, either by differentiating the resulting electrostatic potential or by direct evaluation, depending on the discretization.

\paragraph{Electric field interpolation}
The electric field computed on the mesh is interpolated to the particle positions. To maintain self-consistency, PIC methods typically employ the same
functions for projecting particle charges onto the grid and interpolating the electric field back to the particle positions.

The sequence of operations consisting of charge deposition onto the mesh, computation of the electric field, interpolation of the field to the particle positions, and particle advancement defines one cycle, or time iteration, of PIC algorithms.

In this report, we focus on the particle--mesh interaction steps, namely, charge deposition and electric-field interpolation. We detail these operations for the SGCT-PIC, and HSG-PIC methods, and refer the reader to~\cite{deluzet22,deluzet26} for a description of the remaining steps.

\subsubsection{SGCT-PIC particle--mesh coupling}
\paragraph{Spatial mesh}
The SGCT-PIC method uses a family of independent, anisotropic Cartesian grids, referred to as component grids and defined by~\cref{eq:omega_h}. The mesh is composed of a selection of component grids lying on the $d$ outer diagonals of level $n,\ldots,n+d-1$ of the level lattice as depicted in~\cref{fig:sgct} and defined as
\begin{align}
\label{eq:sgct:levels}
\Omega_h = \bigcup_{\l \in \mathcal{L}_h} \Omega_{h_{\l}}, \quad \mathscr{L}_h &:= \left\{ \bm{\ell}\in\mathbb{N}^d \;\middle|\; n \leq |\bm{\ell}|_1 \leq n+d-1\right\}, \qquad n= |\log h|.
\end{align} 
In two dimensions, the number of component grids and total mesh nodes are given by $2n-1$ and $(3n-1)2^n$, respectively. More generally, these quantities scale as $\mathcal{O}(n^{d-1})$ and $\mathcal{O}(n^{d-1}2^n)$, respectively.

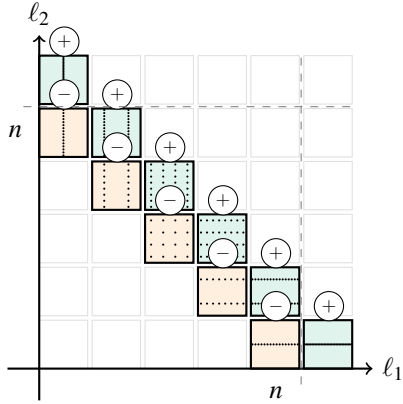
\begin{figure}[h!]
\centering
\centering
\begin{tikzpicture}[scale=.7]
  \def\n{4}
  \def\cellsize{0.915}
  \foreach \i in {0,...,5}{
    \foreach \j in {0,...,5}{
      \draw[draw=black!12,thin] (\i,\j) rectangle ++(\cellsize,\cellsize);
    }
  }
  \foreach \i/\j in {0/4,1/3,2/2,3/1,4/0}{
    \draw[fill=orange1!20,draw=black,thick] (\i,\j) rectangle ++(\cellsize,\cellsize);
    \fullmeshdots{\i}{\j}{\i}{\j}{\cellsize}
  }
  \foreach \i/\j in {0/5,1/4,2/3,3/2,4/1,5/0}{
    \draw[fill=green1!20,draw=black,thick] (\i,\j) rectangle ++(\cellsize,\cellsize);
    \fullmeshdots{\i}{\j}{\i}{\j}{\cellsize}
  }
  \foreach \i/\j in {0/4,1/3,2/2,3/1,4/0}{
    \node[circle,draw,fill=white,inner sep=0pt,minimum size=11pt] at (\i+0.475,\j+1.18) {\scriptsize $-$};
  }
  \foreach \i/\j in {0/5,1/4,2/3,3/2,4/1,5/0}{
    \node[circle,draw,fill=white,inner sep=0pt,minimum size=11pt] at (\i+0.475,\j+1.18) {\scriptsize $+$};
  }
  \draw[->,thick] (-0.6,0) -- (6.3,0) node[right] {$\ell_1$};
  \draw[->,thick] (0,-0.6) -- (0,6.3) node[above] {$\ell_2$};
  \draw[dashed,black!50] (4.95,-0.3) -- (4.95,5.95);
  \draw[dashed,black!50] (-0.3,4.95) -- (5.95,4.95);
  \node at (4.475,-0.45) {$n$};
  \node at (-0.45,4.475) {$n$};
\end{tikzpicture}
\caption{Schematic representation in the two-dimensional case of the component grids defined by \cref{eq:sgct:levels}. Each cell of the level lattice $(\ell_1,\ell_2)$ shows the full nodal mesh of the corresponding anisotropic grid, of resolution $h_{\bm{\ell}}=(2^{-\ell_1},2^{-\ell_2})$, i.e., $2^{\ell_1}\times 2^{\ell_2}$ nodes.}
\label{fig:sgct}
\end{figure}

\paragraph{Charge density deposition}
For each component grid, a regularized Monte-Carlo estimator is defined by convolving the raw density estimator with a localization kernel, also referred to as a shape function. After simplifications, it reads as
\[
\rho_{h_{\l}}(\x,t^\kappa) = \sum_{s=1}^N w_s W^p_{h_{\l}}\big(\x-\x_s(t^\kappa)\big), \quad \text{where}~~ W^p_{h_{\bm{\ell}}}(\bm{x}) = \prod_{i=1}^{d} \mathcal{W}^p_{h_{\ell_i}}(x_i)
\] 
are the shape functions, constructed as tensor products of one-dimensional shape functions, and locally supported on the anisotropic mesh size. Typical choices of shape functions are tensor products of univariate functions~\cite{hockney88,birdsall18}, including piecewise constant functions (nearest-grid-point), piecewise linear functions (cloud-in-cell), or higher-order piecewise polynomials (quadratic B-splines for the triangular-shape-cloud).

Because the shape functions have compact support, each particle contributes to only $p+1$ mesh nodes in each spatial direction on each component grid. Hence, the arithmetic complexity of the charge-deposition step scales as $\mathcal{O}(p^d n^{d-1}N)$.

\paragraph{Electric field interpolation}
A continuous electric field interpolant is reconstructed for each component grid by combining the nodal values obtained by solving the Poisson equation, denoted by $\left(\hat{\bm{E}}^\kappa_{h_{\bm{\ell}}}\right)_{\bm{j}}$, as 
\begin{align}
    {\bm{E}}^\kappa_{h_{\bm{\ell}}}(x) := \sum_{\bm{j} \in I_{h_{\bm{\ell}}}}\left(\hat{\bm{E}}^\kappa_{h_{\bm{\ell}}}\right)_{\bm{j}} \tilde{W}^p_{h_{\bm{\ell}}}\big(\bm{x}-\bm{x}_{\bm{j}}\big),
\end{align}
  where the interpolation shape functions are defined by:
 \[
 \tilde{W}^p_{h_{\l}} = \left(\prod_{i=1}^d h_{\ell_i}\right) W^p_{h_{\l}}.
 \]

The component-grid fields are then recombined at the particle positions according to the sparse-grid combination-technique formula:
\begin{equation}
\label{eq:sgct:combi}
\bm{E}_{h}^{\mathcal{C}}(\bm{x}_s) :=
\sum_{k=0}^{d-1}(-1)^k\binom{d-1}{k}
\sum_{|\bm{\ell}|_1 = n+(d-1)-k} \bm{E}_{h_{\bm{\ell}}}(\bm{x}_s),
\end{equation}
which combines the fields obtained on each of the $d$ diagonals of \cref{eq:sgct:levels} with alternating signs and binomial weights, as depicted by~\cref{fig:sgct}.  This step requires to combine all component-grid contributions for each particle, totaling $\O(p^2n^{d-1}N)$ arithmetic operations.

\subsubsection{HSG-PIC particle--mesh coupling}
Unlike the SGCT-PIC method, which combines independent solutions computed on a collection of full anisotropic component grids, the HSG-PIC method~\cite{deluzet26,deluzet26-1} constructs a single approximation space as a truncated hierarchical sum of small, local subspaces, ordered according to their contribution to the approximation error. The field equation and density estimator are then formulated variationally and solved using a Galerkin method on this approximation space.

\paragraph{Spatial mesh}
The sparse-grid mesh is constructed from a finite-dimensional approximation space defined as a truncated tensor product of univariate B-spline spaces.

For $p\in\mathbb{N}$ and a level $\ell\in\mathbb{N}$ with associated mesh size $h_{\ell}:=2^{-\ell}$, we consider the space of univariate B-splines of degree $p$ and regularity $\mathrm{C}^{p-1}$, which admits a nodal basis denoted by $\{\varphi^p_{h_{\ell},j}\}_{j\in I_{h_\ell}}$. From this univariate B-spline space, we can define the hierarchical subspaces, or increments, by
\begin{align*}
  W_{h_{\l}}(\Omega) = \mathrm{span}\left\{\varphi^p_{h_{\l},\j} \;\middle|\; \j\in J_{h_{\l}}\right\}, \qquad  \varphi^p_{h_{\l},\j}(\x):=\bigotimes_{i=1}^d\varphi^p_{h_{\ell_i},j_i}(x_i),
\end{align*}
where the hierarchical index set is defined by
\begin{align*}
J_{h_{\l}} := J_{h_{\ell_1}}\times\cdots\times J_{h_{\ell_d}}\subset I_{h_{\l}},\qquad
J_{h_{\ell_i}} := \{j\in I_{h_{\ell_i}} \mid j \text{ odd}\}.
\end{align*}
The $L^2$-based sparse-grid approximation space that defines the spatial mesh is defined by
\begin{align*}
V_{h}(\Omega):= \bigg\{ v = \sum_{|\bm{\ell}|_1\leq n} w_{h_{\bm{\ell}}},
\quad \text{where} ~ w_{h_{\bm{\ell}}} \in W_{h_{\bm{\ell}}}(\Omega)
\bigg\},
\end{align*}
and represented in~\cref{fig:hsg_spaces}.

\begin{figure}[h!]
\centering
\begin{tikzpicture}[scale=0.75]
  \def\n{4}
  \def\cellsize{0.893}
  \foreach \i in {0,...,\n}{
    \foreach \j in {0,...,\n}{
      \pgfmathtruncatemacro{\s}{\i+\j}
      \ifnum\s>\n
        \draw[draw=black!25,thin] (\i,\j) rectangle ++(\cellsize,\cellsize);
      \else
        \draw[fill=green1!22,draw=black,thick] (\i,\j) rectangle ++(\cellsize,\cellsize);
        \meshdots{\i}{\j}{\i}{\j}{\cellsize}
      \fi
    }
  }
  \draw[->,thick] (-0.6,0) -- (\n+1.3,0) node[right] {$\ell_1$};
  \draw[->,thick] (0,-0.6) -- (0,\n+1.3) node[above] {$\ell_2$};
  \draw[dashed,black!50] (\n+\cellsize,-0.3) -- (\n+\cellsize,\n+\cellsize);
  \draw[dashed,black!50] (-0.3,\n+\cellsize) -- (\n+\cellsize,\n+\cellsize);
  \node at (\n+0.45,-0.5) {$n$};
  \node at (-0.5,\n+0.45) {$n$};
\end{tikzpicture}
\caption{Schematic representation in the two-dimensional case, for a fixed target level $n$, of the hierarchical subspaces $W_{h_{\bm{\ell}}}(\Omega)$ selected in the sparse-grid $\mathrm{L}^2$-based approximation space. Each filled cell of the level lattice $(\ell_1,\ell_2)$ shows the new mesh points introduced by the corresponding hierarchical increment, whose number $m(\ell_i)=2^{\ell_i-1}$ (and $m(0)=1$) per direction increases with the level. The space retains the lower-diagonal levels $|\bm{\ell}|_1\leq n$.}
\label{fig:hsg_spaces}
\end{figure}
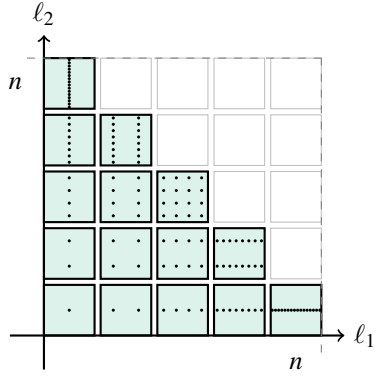

\paragraph{Charge density deposition}
The right-hand side of the Poisson equation is defined through the variational formulation
\[
\left( \Pi_h \rho_N, v_h \right)_{\mathrm{L}^2(\Omega)}
=
\left( \rho_N, v_h \right)_{\mathrm{L}^2(\Omega)},
\qquad
\forall v_h \in V_h(\Omega),
\]
where $\Pi_h$ denotes the Galerkin projection operator onto the approximation space $V_h(\Omega)$. Testing this formulation against the basis functions of $V_h(\Omega)$ and exploiting the shifting property of the Dirac delta distribution, the corresponding right-hand side, denoted by $b_{\l,\j}$, is given by
\[
b_{\l,\j}
=
\sum_{s=1}^N w_s \varphi_{h_{\l,\j}^p}(\x_s).
\]

\paragraph{Electric field interpolation}
The electric field is evaluated at the particles' position
by direct differentiation of the discrete potential, obtained by solving the Poisson equation using a Galerkin method on the approximation space. It reads
\begin{equation}
\label{eq:field_eval}
\bm{E}_{h}^{\mathcal{H}}(\bm{x}_s)
= -\sum_{|\bm{\ell}|_1\leq n}
\sum_{\bm{j}\in J_{h_{\bm{\ell}}}}
\beta_{\bm{\ell},\bm{j}}\,
\bm{\nabla}\varphi_{h_{\bm{\ell}},\bm{j}}^p(\bm{x}_s),
\end{equation}
i.e., the gradient is computed by
differentiating each hierarchical basis function analytically and
summing its contribution, weighted by the hierarchical surplus, denoted by $\beta_{\bm{\ell},\bm{j}}$ and obtained by solving the Poisson equation. 

\begin{remark}
The arithmetic complexities of the different steps of the standard-, SGCT-, and
HSG-PIC algorithms are summarized in~\cref{tab:3} for the two-dimensional case. 
The computational cost of the SGCT- and HSG-PIC methods is dominated by the two
particle--grid operations, namely, charge deposition and field interpolation.
Owing to the noise-reduction properties of sparse-grid approximations, the
number of particles can be reduced for the sparse-grid methods while maintaining
a comparable noise level to that of the standard method. For a fixed ratio of
particles per cell, $P_c$, the total number of particles scales as
$\mathcal{O}(P_c2^{2n})$ for the standard method and as
$\mathcal{O}(P_cn2^n)$ for the sparse-grid methods.
\end{remark}

\begin{table}[h]
\centering
\caption{Arithmetic complexity, measured in terms of the number of arithmetic
operations, of the standard-, SGCT-, and HSG-PIC algorithms in the two-dimensional case.}
\label{tab:3}
\begin{tabular}{lll}
\hline
Operation                            & PIC method & Arithmetic complexity \\ \hline
\multirow{2}{*}{Charge deposition}   & Standard   & $\O(p^2N)$    \\
                                     & SGCT/HSG   & $\O(p^2nN)$  \\ \hline
\multirow{2}{*}{Field interpolation} & Standard   & $\O(p^2N)$    \\
                                     & SGCT/HSG   & $\O(p^2nN)$  \\ \hline
\multirow{2}{*}{Field computation}   & Standard   & $\O(2^{2n})$          \\
                                     & SGCT/HSG   & $\O(pn2^n)$           \\ \hline
\multirow{2}{*}{Push particles}      & Standard   & $\O(N)$       \\
                                     & SGCT/HSG   & $\O(N)$      
\end{tabular}
\end{table}

\section{Hierarchical algorithms for particle--mesh coupling}
 \label{sec:method}
In this section, we introduce two algorithms that compute efficiently the interaction steps coupling the particles and the sparse-grid mesh, namely charge deposition and electric-field interpolation.
We present the algorithms for the two-dimensional SGCT-PIC method. The presentation naturally extends to three-dimensional configurations and to the HSG-PIC method, substituting the shape functions by the basis functions, the component grids by the hierarchical subspaces, etc.

The method is strongly inspired from the hierarchical methods used to accelerate simulations arising in $N$-body problems, such as the FMM. Specifically, the method share the same skeleton than the FMM, but has different component steps. Because of the common structure, we choose to name the steps of our algorithms according to the FMM terminology.  

\paragraph{Main steps}
We first briefly introduce the different steps of the algorithms before describing them in more details in the remainder of this section. The charge density deposition algorithm computes the interactions between sources (particles) and targets (mesh nodes). Its main steps are as follows:

\begin{enumerate}[label=(\roman*)]
\item Particle-to-multipole (P2M): The source information carried by the particles is aggregated into moments, which are computed at the leaves of a directed acyclic graph (DAG) of particle-populated boxes. These moments contain all the particle information required to construct multipole expansions of the charge density. Since the multipole expansions exactly represent the kernel, i.e., the shape function, over the entire domain, they also serve as local expansions.
\item Multipole-to-multipole (M2M): The particle moments are aggregated through the hierarchy of the DAG, from the leaves to the root, in an upward pass. The aggregation is performed in direction-wise sweeps: first a unidirectional $x$-sweep, followed by several unidirectional $y$-sweeps. 
\item Local-to-grid (L2G): The interactions between the boxes and the mesh nodes are computed by forming the local expansions from the particle moments. Owing to the structure of the DAG, only a limited number of boxes interact with each mesh node, resulting in a small number of nonzero interactions.
\end{enumerate}

The electric field interpolation algorithm computes the reverse interactions between sources (mesh nodes) and targets (particles). It consists of the following steps:
\begin{enumerate}[label=(\roman*)]
\item Grid-to-local (G2L): The electric field contributions from the component grids are gathered into local expansions associated with their cells. Since these expansions exactly represent the field over the whole domain, they can also be used as multipole expansions.
\item Local-to-local (L2L):  The local expansions are accumulated in a downward sweep from the component-grid cells to the leaf boxes of the DAG. At the leaves, the expansions from the different component grids are linearly combined according to the combination technique.
\item Local-to-particles (L2P): The recombined local expansions are evaluated at the particle positions to reconstruct the electric field. 
\end{enumerate}

\subsection{Hierarchical particle structure}
\label{sec:particle_dag}
The two algorithms rely on a directed acyclic graph (DAG) of particle-populated boxes at different resolutions. The DAG is constructed from a multilevel dyadic partition of the domain, analogous to that underlying the sparse-grid
mesh, but with boxes instead of mesh nodes. Unlike the tree structure commonly used in FMM, the DAG allows boxes to be coarsened independently in each direction. Consequently, a child box may have multiple parent boxes, resulting in a DAG rather than a tree structure.

 The starting point of the construction of the hierarchical DAG is the definition of boxes, organized into a multilevel hierarchy of different resolutions. For a multi-index $\l = (\ell_1,\ell_2) \in \{1,\dots,n\}^2$, corresponding to the dyadic mesh size $h_{\l}=(2^{-\ell_1},2^{-\ell_2})$,
and a multi-index $\j = (j_1,j_2)\in I_{h_{\l}}$, we 
define the associated dyadic box by
\[
 \omega_{\l,\j}:=[j_1h_{\ell_1},(j_1+1)h_{\ell_1})\times[j_2h_{\ell_2},(j_2+1)h_{\ell_2})\subset \Omega.
\]

\begin{definition}[Level-$\boldsymbol\ell$ layer of the DAG]
The level-$\l$ layer of the DAG is defined as the collection of
the occupied (with particles) boxes at resolution $\l$,
\[
  \mathcal T_{h,\l}
  :=
  \Big\{ \omega_{\l,\j} ~|~
  \exists\, s \in \{1,\dots,N\} \text{ s.t. } \x_s \in \omega_{\l,\j}\Big\}.
\]
We also introduced the particle index set associated to each box, defined by
\[
P(\omega_{\l,\j}) := \{s ~|~ \x_s \in \omega_{\l,\j}\}.
\]
\end{definition}

We denote by $ \omega_{n,\j}$ the boxes corresponding to the finest layer, with $\l=(n,n)$, and called them the leaves of the DAG. 
The hierarchical DAG is formed by the union of the boxes from the level-$\l$ layers, with
\[
\l\in \mathcal{D}_h:=\{(\ell_1,n), ~~ \ell_1=1,\ldots,n\}\cup\mathcal{L}_h,
\]
together with the parent--child relations induced by the successive directional sweeps. 
\begin{definition}[Parent--child relation]
Starting from the leaves on the finest layer, the first sweep successively coarsens the boxes in the $x_1$ direction. Thus, for
$\boldsymbol{\ell}'=\boldsymbol{\ell}-\mathbf e_1$,
a box $\omega_{\boldsymbol{\ell},\mathbf j}$ is a child of
$\omega_{\boldsymbol{\ell}',\mathbf j'}$ if
\[
\omega_{\boldsymbol{\ell},\mathbf j}
\subset
\omega_{\boldsymbol{\ell}',\mathbf j'}.
\]
The resulting $x_1$-coarsened layers then serve as the starting point for independent sweeps in the $x_2$ direction. For each fixed $\ell_1$, and
$\boldsymbol{\ell}'=\boldsymbol{\ell}-\mathbf e_2$,
the same inclusion relation defines the parent--child relation,
\[
\omega_{\l,\j}
\subset
\omega_{\l',\j'}.
\]
\end{definition}

An example of hierarchical DAG is shown in~\cref{fig:dag}. By construction,
the DAG is sparse: only boxes containing at least one particle are instantiated.
Hence, for each level $\boldsymbol{\ell}$, the number of boxes satisfies
\[
\big|\mathcal{T}_{h,\boldsymbol{\ell}}\big|
\leq
\min\big(N,\,2^{|\boldsymbol{\ell}|_1}\big).
\]

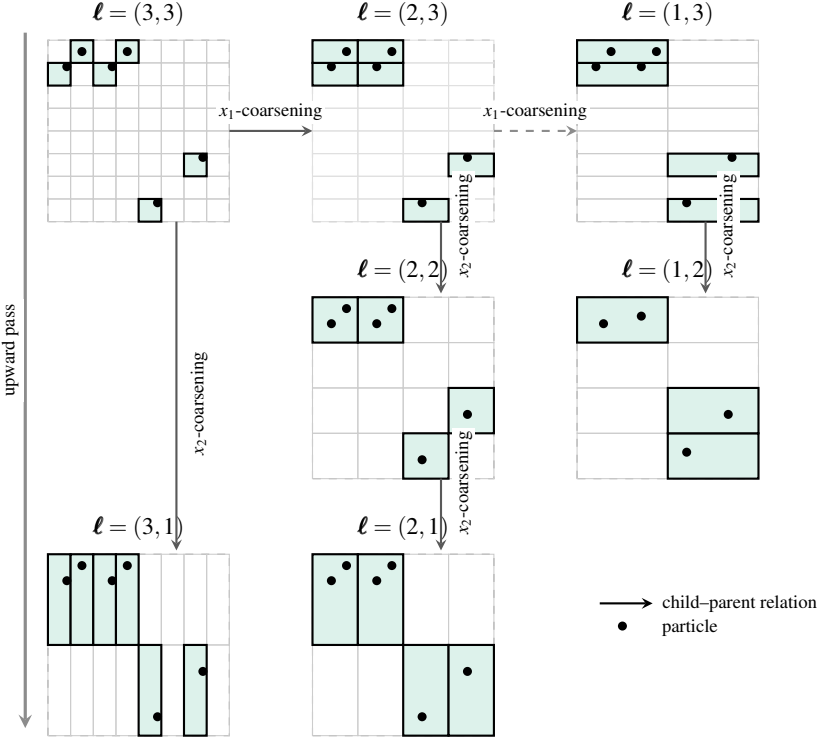
\begin{figure}[h!]
\centering
\begin{tikzpicture}[
    >=stealth,
    cell/.style={
        fill=green1!22,
        draw=black,
        thick
    },
    intermediatecell/.style={
        fill=gray!12,
        draw=black!35,
        dashed,
        thick
    },
    domainbox/.style={
        dashed,
        draw=black!45,
        thin
    },
    intermediatebox/.style={
        dash pattern=on 1pt off 1pt,
        draw=black!25,
        thin
    },
    xarrow/.style={
        ->,
        thick,
        black!65
    },
    yarrow/.style={
        ->,
        thick,
        black!65
    },
    particle/.style={
        circle,
        fill=black,
        inner sep=1.2pt},
    fadedparticle/.style={
        circle,
        fill=black!35,
        inner sep=1.2pt}
]

\def\L{2.4}
\def\DX{3.5}
\def\DY{3.4}


\begin{scope}[xshift=0cm,yshift=2*\DY cm]
\node[font=\small\bfseries] at (1.2,2.75)
    {$\boldsymbol{\ell}=(3,3)$};
\draw[domainbox] (0,0) rectangle (\L,\L);
\foreach \i in {0,...,7} {
    \foreach \j in {0,...,7} {
        \draw[black!20]
            ({0.3*\i},{0.3*\j})
            rectangle ({0.3*(\i+1)},{0.3*(\j+1)});
    }
}
\draw[cell] (0,1.8) rectangle (0.3,2.1);
\draw[cell] (0.3,2.1) rectangle (0.6,2.4);
\draw[cell] (0.6,1.8) rectangle (0.9,2.1);
\draw[cell] (0.9,2.1) rectangle (1.2,2.4);
\draw[cell] (1.2,0) rectangle (1.5,0.3);
\draw[cell] (1.8,0.6) rectangle (2.1,0.9);
\node[particle] at (0.25,2.05) {};
\node[particle] at (0.45,2.25) {};
\node[particle] at (0.85,2.05) {};
\node[particle] at (1.05,2.25) {};
\node[particle] at (1.45,0.25) {};
\node[particle] at (2.05,0.85) {};
\end{scope}

\begin{scope}[xshift=\DX cm,yshift=2*\DY cm]
\node[font=\small\itshape] at (1.2,2.75)
    {$\boldsymbol{\ell}=(2,3)$};
\draw[domainbox] (0,0) rectangle (\L,\L);
\foreach \i in {0,1,2,3} {
    \foreach \j in {0,...,7} {
        \draw[black!12]
            ({0.6*\i},{0.3*\j})
            rectangle ({0.6*(\i+1)},{0.3*(\j+1)});
    }
}
\draw[cell] (0,1.8) rectangle (0.6,2.1);
\draw[cell] (0,2.1) rectangle (0.6,2.4);
\draw[cell] (0.6,1.8) rectangle (1.2,2.1);
\draw[cell] (0.6,2.1) rectangle (1.2,2.4);
\draw[cell] (1.2,0) rectangle (1.8,0.3);
\draw[cell] (1.8,0.6) rectangle (2.4,0.9);
\node[particle] at (0.25,2.05) {};
\node[particle] at (0.45,2.25) {};
\node[particle] at (0.85,2.05) {};
\node[particle] at (1.05,2.25) {};
\node[particle] at (1.45,0.25) {};
\node[particle] at (2.05,0.85) {};
\end{scope}

\begin{scope}[xshift=2*\DX cm,yshift=2*\DY cm]
\node[font=\small\bfseries] at (1.2,2.75)
    {$\boldsymbol{\ell}=(1,3)$};
\draw[domainbox] (0,0) rectangle (\L,\L);
\foreach \i in {0,1} {
    \foreach \j in {0,...,7} {
        \draw[black!20]
            ({1.2*\i},{0.3*\j})
            rectangle ({1.2*(\i+1)},{0.3*(\j+1)});
    }
}
\draw[cell] (0,1.8) rectangle (1.2,2.1);
\draw[cell] (0,2.1) rectangle (1.2,2.4);
\draw[cell] (1.2,0) rectangle (2.4,0.3);
\draw[cell] (1.2,0.6) rectangle (2.4,0.9);
\node[particle] at (0.25,2.05) {};
\node[particle] at (0.85,2.05) {};
\node[particle] at (0.45,2.25) {};
\node[particle] at (1.05,2.25) {};
\node[particle] at (1.45,0.25) {};
\node[particle] at (2.05,0.85) {};
\end{scope}

\begin{scope}[xshift=\DX cm,yshift=\DY cm]
\node[font=\small\bfseries] at (1.2,2.75)
    {$\boldsymbol{\ell}=(2,2)$};
\draw[domainbox] (0,0) rectangle (\L,\L);
\foreach \i in {0,1,2,3} {
    \foreach \j in {0,1,2,3} {
        \draw[black!20]
            ({0.6*\i},{0.6*\j})
            rectangle ({0.6*(\i+1)},{0.6*(\j+1)});
    }
}
\draw[cell] (0,1.8) rectangle (0.6,2.4);
\draw[cell] (0.6,1.8) rectangle (1.2,2.4);
\draw[cell] (1.2,0) rectangle (1.8,0.6);
\draw[cell] (1.8,0.6) rectangle (2.4,1.2);
\node[particle] at (0.25,2.05) {};
\node[particle] at (0.45,2.25) {};
\node[particle] at (0.85,2.05) {};
\node[particle] at (1.05,2.25) {};
\node[particle] at (1.45,0.25) {};
\node[particle] at (2.05,0.85) {};
\end{scope}

\begin{scope}[xshift=2*\DX cm,yshift=\DY cm]
\node[font=\small\bfseries] at (1.2,2.75)
    {$\boldsymbol{\ell}=(1,2)$};
\draw[domainbox] (0,0) rectangle (\L,\L);
\foreach \i in {0,1} {
    \foreach \j in {0,1,2,3} {
        \draw[black!20]
            ({1.2*\i},{0.6*\j})
            rectangle ({1.2*(\i+1)},{0.6*(\j+1)});
    }
}
\draw[cell] (0,1.8) rectangle (1.2,2.4);
\draw[cell] (1.2,0) rectangle (2.4,0.6);
\draw[cell] (1.2,0.6) rectangle (2.4,1.2);
\node[particle] at (0.35,2.05) {};
\node[particle] at (0.85,2.15) {};
\node[particle] at (1.45,0.35) {};
\node[particle] at (2.0,0.85) {};
\end{scope}

\begin{scope}[xshift=0cm,yshift=0cm]
\node[font=\small\bfseries] at (1.2,2.75)
    {$\boldsymbol{\ell}=(3,1)$};
\draw[domainbox] (0,0) rectangle (\L,\L);
\foreach \i in {0,...,7} {
    \foreach \j in {0,1} {
        \draw[black!20]
            ({0.3*\i},{1.2*\j})
            rectangle ({0.3*(\i+1)},{1.2*(\j+1)});
    }
}
\draw[cell] (0,1.2) rectangle (0.3,2.4);
\draw[cell] (0.3,1.2) rectangle (0.6,2.4);
\draw[cell] (0.6,1.2) rectangle (0.9,2.4);
\draw[cell] (0.9,1.2) rectangle (1.2,2.4);
\draw[cell] (1.2,0) rectangle (1.5,1.2);
\draw[cell] (1.8,0) rectangle (2.1,1.2);
\node[particle] at (0.25,2.05) {};
\node[particle] at (0.45,2.25) {};
\node[particle] at (0.85,2.05) {};
\node[particle] at (1.05,2.25) {};
\node[particle] at (1.45,0.25) {};
\node[particle] at (2.05,0.85) {};
\end{scope}

\begin{scope}[xshift=\DX cm,yshift=0cm]
\node[font=\small\bfseries] at (1.2,2.75)
    {$\boldsymbol{\ell}=(2,1)$};
\draw[domainbox] (0,0) rectangle (\L,\L);
\foreach \i in {0,1,2,3} {
    \foreach \j in {0,1} {
        \draw[black!20]
            ({0.6*\i},{1.2*\j})
            rectangle ({0.6*(\i+1)},{1.2*(\j+1)});
    }
}
\draw[cell] (0,1.2) rectangle (0.6,2.4);
\draw[cell] (0.6,1.2) rectangle (1.2,2.4);
\draw[cell] (1.2,0) rectangle (1.8,1.2);
\draw[cell] (1.8,0) rectangle (2.4,1.2);
\node[particle] at (0.25,2.05) {};
\node[particle] at (0.45,2.25) {};
\node[particle] at (0.85,2.05) {};
\node[particle] at (1.05,2.25) {};
\node[particle] at (1.45,0.25) {};
\node[particle] at (2.05,0.85) {};
\end{scope}

\draw[xarrow]
    (2.4,2*\DY+1.2) -- (\DX,2*\DY+1.2);
\node[font=\scriptsize,fill=white,inner sep=1pt]
    at (2.95,2*\DY+1.45)
    {$x_1$-coarsening};

\draw[xarrow,black!45,dashed]
    (\DX+2.4,2*\DY+1.2) -- (2*\DX,2*\DY+1.2);
\node[font=\scriptsize,fill=white,inner sep=1pt]
    at (6.45,2*\DY+1.45)
    {$x_1$-coarsening};


\draw[yarrow]
    (1.7,2*\DY) -- (1.7,2.45);
\node[
    font=\scriptsize,
    rotate=90,
    fill=white,
    inner sep=1pt
]
    at (2.0,\DY+1)
    {$x_2$-coarsening};

\draw[yarrow]
    (\DX+1.7,2*\DY) -- (\DX+1.7,\DY+2.45);
\node[
    font=\scriptsize,
    rotate=90,
    fill=white,
    inner sep=1pt
]
    at (\DX+2.0,2*\DY-0.05)
    {$x_2$-coarsening};

\draw[yarrow]
    (\DX+1.7,\DY) -- (\DX+1.7,2.45);
\node[
    font=\scriptsize,
    rotate=90,
    fill=white,
    inner sep=1pt
]
    at (\DX+2.0,\DY-0.05)
    {$x_2$-coarsening};

\draw[yarrow]
    (2*\DX+1.7,2*\DY) -- (2*\DX+1.7,\DY+2.45);
\node[
    font=\scriptsize,
    rotate=90,
    fill=white,
    inner sep=1pt
]
    at (2*\DX+2.0,2*\DY-0.05)
    {$x_2$-coarsening};

\draw[->,very thick,black!45]
    (-0.3,2*\DY+2.5)
    -- (-0.3,0.1);
\node[
    rotate=90,
    font=\scriptsize,
    fill=white,
    inner sep=2pt
]
    at (-0.5,1.5*\DY)
    {upward pass};

  \draw[->,thick,black] (2*\DX+0.3,1.75) -- (2*\DX+1.0,1.75)
       node[right,font=\scriptsize] {child--parent relation};

  \node[particle] at (2*\DX+0.6,1.45) {};
  \node[right,font=\scriptsize] at (2*\DX+1,1.425) {particle};

\end{tikzpicture}

\caption{Schematic representation of the hierarchical DAG and the directional
sweeps, starting from the finest level $\boldsymbol{\ell}=(3,3)$. A first sweep
coarsens in the $x_1$ direction (horizontal arrows), then, independent
sweeps coarsen in the $x_2$ direction (vertical arrows). Only the level-$\l$ layers with $\l\in\mathcal{D}_h$ and their occupied
boxes are instantiated.}
\label{fig:dag}
\end{figure}

\subsection{Exact polynomial kernel representation}
\label{sec:idea}
The proposed method can be interpreted as an FMM algorithm for computing interactions between a population of particles and a collection of meshes (the sparse-grid mesh, composed of a collection of component grids) through a piecewise-polynomial kernel (the shape function). The central idea of the FMM is to replace the direct evaluation of interactions between well-separated source and target
regions by an equivalent representation based on multipole expansions, which is possible when the kernel is sufficiently smooth in the corresponding far-field region. 
In the present setting, the piecewise-polynomial structure of the kernel leads to several important simplifications compared with the standard FMM. First, the multipole expansions are exact and therefore do not
require any truncation or approximation. Second, because the kernel is piecewise polynomial, these expansions are valid throughout the domain rather
than only in a restricted far-field region. Consequently, the same representation can serve both as a multipole expansion of a source particle cluster and as a local expansion for evaluating its contribution at target mesh nodes. This property eliminates the need for the usual multipole-to-local translation step and substantially simplifies the interaction scheme.

\paragraph{Kernel representation}
The shape functions are defined as tensor products of univariate piecewise polynomials. Consequently, on each sufficiently small spatial region, they admit an exact representation as a finite-term polynomial expansion.
\begin{lemma}[Local polynomial representation]
\label{lem:pol}
The shape functions, restricted to the boxes of associated level, are multivariate polynomial of partial degree $p$ in each coordinate, i.e., 
\[
(\psi_{h_{\l},\j}^p)_{|\omega_{\l,\j}}\in\spn \left\{1,x_1,x_2, x_1x_2,\ldots,x_1^p,x_2^p,x_1^px_2^p \right\}, \qquad \forall \l\in \mathcal{L}_h,\,\j\in I_{h_{\l}}.
\] 
\end{lemma}

Within each box, the shape functions can therefore be represented exactly by
$(p+1)^d$ moments. For example, the charge density restricted to a single box can be computed as
\[
(\rho_{\l,\j})_{|\omega_{\l,\j}}=  \sum_{\bm{i}\in\{0,\ldots,p\}^2} a_{\bm{i}}(\omega_{\l,\j})\, \sum_{\x_s\in \omega_{\l,\j}} w_s x_{s,1}^{i_1}x_{s,2}^{i_2},
\]
where $a_{\bm{i}}(\omega_{\ell,\bm{j}})$ are the coefficients of the corresponding shape function in the monomial basis. All particle-dependent information is contained in the innermost sum, which can be computed once and reused for all component grids. This reuse, described in the following paragraphs, separates the particle-dependent and grid-dependent computations and reduces the computational complexity of the charge-deposition algorithm from $\mathcal{O}(p^2 nN)$ to $\mathcal{O}(p^2(N + 2^{2n}))$. The same strategy can be applied to the electric field interpolation.

\subsection{Steps of the algorithms}
\label{sec:steps}
\subsubsection{Charge density deposition}
The charge density deposition algorithm computes the accumulation of all the individual particle charges onto the set of component grids.
\paragraph{Particle-to-multipole (P2M)}
The first step of the algorithm is called the particle-to-multipole (P2M) operation and consists of collecting all the required particle information at the leaves of the hierarchical DAG. These informations depend on the structure of the kernel, which in our case can be exactly represented by a finite number of raw particle moments.

\begin{definition}[Raw particle moments]
\label{def:mom}
For a box $\omega_{\l,\j}\in \mathcal{T}_{h,\l}$, the raw particle moments, indexed by the order $\bm{i}=(i_1,i_2)\in \{0,\ldots,p\}^2$, are defined by 
\begin{align}
\label{eq:raw_mom}
\mathcal{M}_{\bm{i}}(\omega_{\l,\j}):= \sum_{s\in P(\omega_{\l,\j})} w_s \,x_{s,1}^{i_1}x_{s,2}^{i_2}.
\end{align}
\end{definition}
We also define the set of moments of a layer by the set of box moments at this layer,
\[
\mathcal{M}_{\bm{i}}(\mathcal{T}_{h,\l})=\left\{\mathcal{M}_{\bm{i}}(\omega_{\l,\j})~|~ \j~s.t.~\omega_{\l,\j}\in\mathcal{T}_{h,\l}\right\}.
\]

Since the boxes within a given layer are disjoint, each particle belongs to a unique box at that layer. The P2M operation therefore consists
of computing, for each particle, the $(p+1)^2$ moments associated with the unique leaf box containing that particle. This requires $\mathcal{O}(p^2N)$ arithmetic operations.

\paragraph{Multipole-to-multipole (M2M)}
The following step is called the multipole-to-multipole (M2M) operation and consists of aggregating the moments in the DAG hierarchy. This operation is performed by successively applying unidirectional aggregation operator, direction after direction, to the particle moments from the leaves to the root, in an upward pass. This upward pass is also called a sweep, which is defined direction-wise.

\begin{definition}[Unidirectional aggregation]
For levels $\l'\leq\l$ with $\l'=\l-\bm{e}_k$ for a direction $k=1,2$, the unidirectional aggregation operator in that direction is denoted by 
\[
\mathcal{S}_{k,\l,\l'}: \mathcal{M}_{\bm{i}}(\mathcal{T}_{h,\l})\mapsto \tilde{\mathcal{M}}_{\bm{i}}(\mathcal{T}_{h,\l'}),
\]
and defined by the sum of the moments from the children boxes in the direction $k$, i.e.,
\begin{align}
\label{eq:parent_moment}
\tilde{\mathcal{M}}_{\bm{i}}(\omega_{\l',\j'}) = \sum_{\substack{\omega \in \mathcal{T}_{h,\l}\\ \omega \subset \omega_{\l',\j'}}}\mathcal{M}_{\bm{i}}(\omega), \qquad \text{for}~~\omega_{\l',\j'}\in \mathcal{T}_{h,\l'}.
\end{align}
\end{definition}

The moments of a parent box, aggregated from its children according to~\cref{eq:parent_moment}, coincide with the raw moments computed directly from the particles contained in the parent box, as defined in~\cref{eq:raw_mom}. This exact preservation of the raw moments under aggregation is formalized in~\Cref{prop:coarseop}.

\begin{proposition}
\label{prop:coarseop}
The unidirectional aggregation preserves the raw moments, i.e., for levels $\bm{\ell}'\leq\bm{\ell}$ with
$\bm{\ell}'=\bm{\ell}-\bm{e}_k$,
\[
\widetilde{\mathcal{M}}_{\bm{i}}(\omega_{\bm{\ell}',\bm{j}'})
=
\mathcal{M}_{\bm{i}}(\omega_{\bm{\ell}',\bm{j}'}),
\qquad
\forall\,\omega_{\bm{\ell}',\bm{j}'}\in\mathcal{T}_{h,\bm{\ell}'}.
\]
\end{proposition}
\begin{proof}
Since each component of $\tilde{\mathcal{M}}_{\bm{i}}(\omega_{\l',\j'})$ is an additive functional of the particle set, which is a sum over the particles of weights times a fixed global monomial evaluated at the particle's position, and the parent box is a disjoint union of the children boxes, the moment of the union is the sum of the moments of the parts, independent of any choice of expansion center.
\end{proof}

A direct implication of~\cref{prop:coarseop} is that the M2M operator is exact, without the need of a re-centering of the moments. This is different from the traditional FMM M2M operator, which translates an original multipole expansion, associated to a parent box, into a new valid one by shifting from the center of the child box to the center of the parent box. An untruncated multipole expansion of a generic kernel about a fixed global origin would not, in general, remain low-dimensional at all levels in the hierarchy, so that shifting towards new centers is required in FMM. The FMM translation operator can be avoided here because the kernel, restricted to the DAG boxes, is a multivariate polynomial that can be represented exactly with uncentered moments, computed about a fixed global origin rather than a per-cluster local center, of a finite-dimensional monomial space. 

Another consequence of the exact representation of the kernel over the whole domain is that the multipole expansion needs not to be translated into a local expansion centered at the target box. In other words, the multipole-to-local (M2L) operation is unnecessary in our case, since the multipole expansion itself directly provides the required local representation in the target box.

The exactness of the aggregation is crucial for avoiding the $\O(p^2 nN)$ complexity that would otherwise arise from processing each particle on each component grid. Instead,
the moments are computed at the leaves with a complexity of $\O(p^2N)$ through the P2M operation and subsequently aggregated
across the hierarchy through the M2M operation. At each level, the cost of the M2M operation is proportional to the number of occupied boxes, so that the total cost of the
aggregation depends on the total number of occupied boxes in the hierarchical representation. The total number of occupied boxes across all layers, including the intermediate ones, is given by
\begin{align}
\label{sum_1}
S_{n,N}=\sum_{k=n+1}^{2n}\min(N,2^{k}) + (n-1)\min(N,2^n)+n\min(N,2^{n+1}). 
\end{align}
Since for a given particle-per-cell ratio $P_c$, the total number of particles verifies
\[
N=P_c(3n-1)2^n > 2^{n+1},
\]
the last two terms of~\cref{sum_1} scales as $\O(n2^n)$ and are negligible. The value of the first term depends on the particle-per-cell ratio and the mesh resolution. For the configurations of interest, namely $n\le 10$ and $P_c\ge 35$, the minimum is equal to $2^k$ for all $k$, and hence the first term is proportional to $2^{2n}$. Consequently, the arithmetic cost of the M2M operator is proportional to $p^22^{2n}$.

\paragraph{Local-to-grid (L2G)}
The last step of the charge deposition algorithm is the local-to-grid (L2G) operation, which consists of reconstructing the charge density on the sparse-grid mesh from the aggregated moments in the boxes. The L2G operation computes the interactions between all the boxes and the mesh nodes. Because the structure of the DAG is inherited from the sparse-grid mesh, some of these interactions contribute exactly zero to the charge density. We define an admissibility criterion that discard all these zero-contribution interactions.

\begin{definition}[Admissible criterion]
Let $\omega_{\l',\j'}\in \mathcal{T}_{h,\l'}$ be a box and $\x_{\l,\j}\in \Omega_h$ be a mesh node, then the pair $(\x_{\l,\j},\omega_{\l',\j'})$ is admissible if and only if
\begin{align}
\label{eq:adm_crit}
\l = \l' ~~\text{and} ~~ j_1=j_1'+k_1, ~ j_2=j_2'+k_2, ~~\text{with}~\bm{k}\in\{-\lfloor p/2 \rfloor,\ldots,\lfloor p/2 \rfloor+1\}^2.
\end{align}
\end{definition}

The admissibility criterion~\cref{eq:adm_crit} is an exact set-containment test, which follows directly from the DAG structure, whose hierarchy and dyadic
refinement are inherited from the sparse-grid mesh. Every admissible pair therefore contributes exactly to the density, while every nonadmissible pair
contributes exactly zero. This is in contrast to the smooth-kernel setting of the standard FMM, where far-field interactions are approximated by truncated
multipole or local expansions and are therefore only negligibly small up to the chosen approximation accuracy.

For a given sparse-grid mesh node of index $({\l,\j})$ we denote the set of admissible boxes by 
\[
\mathcal{A}(\x_{\l,\j}):=\left\{ \omega ~~\text{s.t. the pair } (\x_{\l,\j},\omega)~\text{is admissible} \right\}.
\]

The charge density deposition can then be performed by accumulating for each mesh node, the interactions of all the admissible box--node pairs.

\begin{proposition}
The charge density accumulated on the sparse-grid mesh node of index $(\l,\j)$ can be expressed as
\[
\rho_{\l,\j}=\sum_{\omega\in\mathcal{A}(\x_{\l,\j})}\sum_{\bm{i} \in \{0,\ldots,p\}^2} a_{\bm{i}}(\omega)\mathcal{M}_{\bm{i}}(\omega),
\]
where $a_{\bm{i}}(\omega)$ are the coefficients of the basis functions' restriction to the box $\omega$ in the monomial basis $\{1,x_1,x_2,\ldots\}$.
\end{proposition}
\begin{proof}
The result follows from~\cref{lem:pol} and~\cref{def:mom}.
\end{proof}

The L2G step plays the combined role of the local-to-local (L2L) and local-to-particle (L2P) operations in standard FMM algorithms for smooth kernels. In the present setting, however, the target of each interaction is
not a box containing a cluster of particles, but a single sparse-grid mesh node. Consequently, there is no need to propagate local expansions from the root toward the leaves of the DAG, as is required by the L2L operation in the
standard FMM. The local expansion can instead be evaluated directly at the target mesh nodes.

For each mesh node, the number of admissible boxes depends only on the degree of the basis functions and is given by 
\[
|\mathcal{A}(\x_{\l,\j})|=4(\lfloor p/2\rfloor+1)^2,
\]
so that the number of operations required to compute the charge density at a node is $4(\lfloor p/2\rfloor+1)^2(p+1)^2$, and the total number of arithmetic operations of the L2G step scales as
$\O(p^4 n2^n).$

\subsubsection{Electric field interpolation}
The electric field interpolation algorithm maps back the electric field contributions from the component grids to the particle positions using the combination technique.
\paragraph{Grid-to-local (G2L)}
In the SGCT-PIC method, the functions used to deposit the charge density on the component grids are the same than those employed to interpolate the electric field at the particles's positions. 
As a result,~\cref{lem:pol} can also be applied to decompose the electric field in the monomial basis $\{1,x_1,x_2,\ldots\}$ as
\begin{align}
\label{eq:locE}
(\E_{\l}(\x_s))_{|\omega_{\l,\j}}=  \sum_{\bm{i}\in\{0,\ldots,p\}^2} \bm{b}_{\bm{i}}(\omega_{\l,\j})x_{s,1}^{i_1}x_{s,2}^{i_2},
\end{align}
where $\bm{b}_{\bm{i}}(\omega_{\l,\j})$ are the coefficients of the electric field contribution on a given component grid, restricted to the box, in the monomial basis. \Cref{eq:locE} provides a local, or equivalently, a multipole expansion of each component grid electric field contribution, which is valid over the whole domain. 

The grid-to-local (G2L) operator consists of computing the coefficients of the local expansion~\cref{eq:locE}. These coefficients must be computed for each box associated with the component-grid nodes. Consequently, a total of $(p+1)^2(3n-1)2^n$ coefficients must be computed. The computational cost of the G2L step therefore scales with the total number of sparse-grid mesh nodes as $\O(p^2 n 2^n)$.

\paragraph{Local-to-local (L2L)}
The local-to-local (L2L) operator consists of accumulating the coefficients of the local expansions~\cref{eq:locE} in a downward sweep of the DAG, aggregating expansions from component grid cells to the leaf boxes. The L2L sweep pattern is not a direction-wise sweep like the M2M one, as it only aggregates the expansions from cells associated to a component-grid level directly to the leaves. The coefficients of the local expansions are first aggregated through the DAG
and then linearly combined according to the combination technique formula~\cref{eq:sgct:combi}. The resulting L2L operation is defined as follows.
\begin{definition}For each leaf box, we define the coefficients of the combination-technique local expansion, for $\bm{i}=(i_1,i_2)\in \{0,\ldots,p\}^2$, by  
\begin{align}
\label{eq:comb_coeff}
\tilde{\bm{b}}_{\bm{i}}(\omega_{n,\j}):= \sum_{k=0}^{1}(-1)^k
\sum_{\substack{|\bm{\ell}|_1 = n+1-k}} \;\; \sum_{\substack{\omega \in \mathcal{T}_{h,\l}\\ \omega_{n,\j} \subset \omega}} \bm{b}_{\bm{i}}(\omega).
\end{align}
\end{definition}

For each occupied leaf box, the local expansions from all the component grids have to be combined, which results in a cost of $\O(p^2n\min(N,2^{2n}))$ operations. For the regimes of interest, i.e., $n\le10$ and $P_c\ge35$, the cost is proportional to $p^2n2^{2n}$

\paragraph{Local-to-particle (L2P)}
The local-to-particle (L2P) operation maps the recombined local expansions from the leaf boxes to the particles. The electric field is defined at the particle positions using the recombined local expansion coefficients~\cref{eq:comb_coeff} as 
\begin{align}
\label{eq:comb_elec}
(\tilde{\E}_{h}^{\mathcal{C}}(\x_s))_{|\omega_{n,\j}}= \sum_{\substack{\bm{i}\in \{0,\ldots,p\}^2}}  \tilde{\bm{b}}_{\bm{i}}(\omega_{n,\j})x_{s,1}^{i_1}x_{s,2}^{i_2}.
\end{align}
\Cref{prop:comb_elec} states that the L2L and L2P operators are exact.

\begin{proposition}
\label{prop:comb_elec}
On each leaf box $\omega_{n,\j}$, the electric field obtained from the local
expansions coincides with the recombined electric field:
\[
\left.\tilde{\E}_{h}^{\mathcal{C}}(\x_s)\right|_{\omega_{n,\j}}
=
\left.\E_{h}^{\mathcal{C}}(\x_s)\right|_{\omega_{n,\j}}.
\]
\end{proposition}
\begin{proof}
Substituting~\cref{eq:comb_coeff} into~\cref{eq:comb_elec} and
\cref{eq:locE} into~\cref{eq:sgct:combi}, linearity shows that the two
expressions coincide on each leaf box. Since the leaf boxes form a disjoint
partition of the domain, the result follows.
\end{proof}

Each particle belongs to a unique leaf box, so only $(p+1)^2$ coefficients of the local expansion need to be evaluated per particle to reconstruct the electric field at its position. The L2P operation therefore requires
$\mathcal{O}(p^2N)$ arithmetic operations.

\subsection{Arithmetic complexities}
\label{sec:alg}
In this section, we summarize the arithmetic complexities of the two proposed algorithms, which we refer to as the FMM versions, and compare them with those of the corresponding standard versions. The arithmetic complexity, measured as the number of arithmetic operations, of each step of the charge density deposition and electric field interpolation algorithms is reported in~\cref{tab:1}. The total arithmetic complexities of the FMM versions of the charge density deposition and electric field interpolation algorithms are, for configurations with $n\le 10$ and $P_c\ge35$, respectively,
\begin{align}
\label{eq:cost}
\O(p^2(N+2^{2n})) ~~\text{and}~~ \O(p^2(N+n2^{2n})).
\end{align}

\renewcommand{\thefootnote}{\fnsymbol{footnote}}
\begin{table}[t]
\centering
\caption{Arithmetic complexity, measured as the number of arithmetic
operations, of the FMM and standard versions of algorithms.}
\label{tab:1}

\begin{tabular}{lll}
\hline
Version & Operation & Arithmetic complexity \\
\hline
Standard & Charge deposition
& $\O(p^2nN)$ \\
FMM & P2M
& $\O(p^2N)$ \\
& M2M
& $\O(p^22^{2n})$\textsuperscript{$*$}\\
& M2G
& $\O(p^4n2^n)$ \\
\hline
Standard & Field interpolation
& $\O(p^2nN)$ \\

FMM & G2L
& $\O(p^2n2^n)$ \\
& L2L
& $\O(p^2n2^{2n})$\textsuperscript{$*$} \\
& L2P
& $\O(p^2N)$ \\
\hline
Both & Push particles & $\O(N)$ \\
\hline
Both & Field computation & $\O(pn2^n)$ \\
\end{tabular}
\\ 
\smallskip
{\textsuperscript{*}\footnotesize For the regimes of interest, namely
$n\le 10$ and $P_c\ge35$.}
\end{table}

\begin{figure} [hbt!]
 \begin{minipage}[]{0.42\textwidth}
 \includegraphics[width=\textwidth]{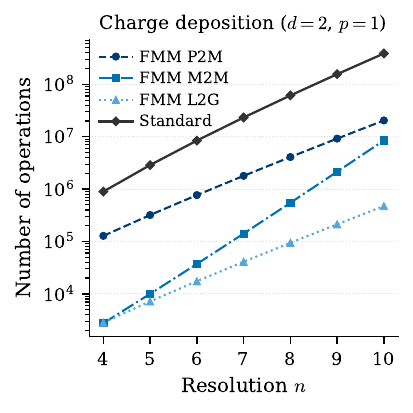}   
\end{minipage} 
  \begin{minipage}[]{0.44\textwidth}
 \includegraphics[width=\textwidth]{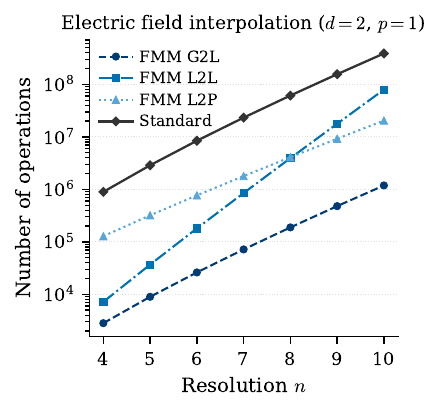}   
\end{minipage} \\
 \begin{minipage}[]{0.42\textwidth}
 \includegraphics[width=\textwidth]{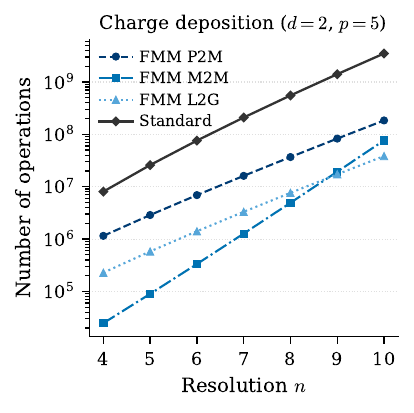}   
\end{minipage} 
  \begin{minipage}[]{0.44\textwidth}
 \includegraphics[width=\textwidth]{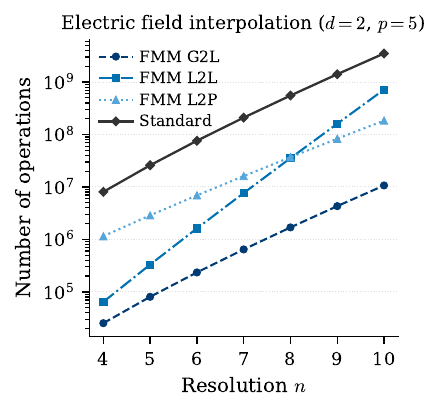}   
\end{minipage}
 \caption{Arithmetic complexity of the charge deposition and electric field interpolation FMM- and standard-version algorithms of the SGCT-PIC method as a function of the mesh resolution, for a ratio of $500$ particles per cell.} \label{fig:1}
\end{figure}

The leading terms in the estimates of~\cref{eq:cost} depend on the number of particles per cell. For a large particle-per-cell ratio, the complexities are dominated by the term $\O(p^2N)$.  Conversely, for a small particle-per-cell ratio, the terms $\O(p^22^{2n})$ and $\O(p^2n2^{2n})$ may become dominant. In practice, the particle-per-cell ratio typically ranges from a few tens to few thousands, depending on the application, with a few hundred particles per cell being representative.

The arithmetic complexity estimates are plotted in~\cref{fig:1} as functions of the mesh resolution for several shape-function degrees, with a fixed ratio of
$500$ particles per cell. The FMM-version charge-deposition algorithm reduces the arithmetic complexity by approximately one order of magnitude. For electric field interpolation, the FMM version also requires fewer operations, but
the reduction is less pronounced than for charge deposition.

\section{Numerical results}
\label{sec:num}
We restrict our numerical study to the two-dimensional case and sequential executions on a single CPU core. We consider the diocotron instabilty configuration described in~\cref{apd:1} of the appendix. All experiments are performed on a single Apple M4 core with $32\,\mathrm{GB}$ of RAM. The methods are implemented in a code developed in Python and C++\footnote{The code is available at \url{https://gitlab.inria.fr/cguillet/sg-pic}}. The linear systems are solved using a Cholesky factorization from the \texttt{scipy.sparse} library.

All simulations performed with the FMM and standard versions of the algorithms produce identical results up to roundoff errors. The execution times of the different components of one time iteration, averaged over 10 iterations, of the SGCT- and HSG-PIC methods are reported in~\cref{tab:2} and~\cref{tab:4}, respectively. 
The experiments consider two mesh resolutions, $n=8$ and $n=10$, and three particle-to-cell ratios, $P_c=50$, $P_c=500$, and $P_c=2{,}500$, using piecewise-linear shape functions throughout.

\begin{figure}[h!]
\centering
 \begin{minipage}[]{0.49\textwidth}
 \includegraphics[width=\textwidth]{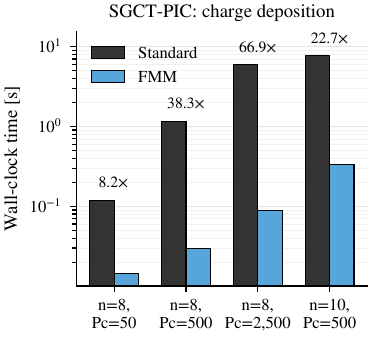}   
\end{minipage} 
 \begin{minipage}[]{0.49\textwidth}
 \includegraphics[width=\textwidth]{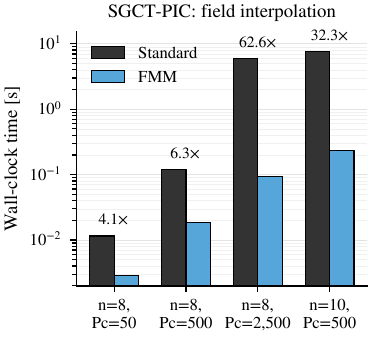}  
\end{minipage} \\ \medskip 
\begin{tabular}{llllll}
\hline
\multirow{2}{*}{Version} & \multirow{2}{*}{Operation} & \multicolumn{4}{c}{Time execution {[}s{]}}                                                                                                                                                                                                                            \\
                         &                            & \begin{tabular}[c]{@{}l@{}}$n=8$\\ $P_c=50$\end{tabular} & \begin{tabular}[c]{@{}l@{}}$n=8$\\ $P_c=500$\end{tabular} & \begin{tabular}[c]{@{}l@{}}$n=8$\\ $P_c=2{,}500$\end{tabular} & \multicolumn{1}{c}{\begin{tabular}[c]{@{}c@{}}$n=10$\\ $P_c=500$\end{tabular}} \\ \hline
Standard                 & Charge deposition          & $1.18\mathrm{E}{-01}$                                    & $1.15\mathrm{E}{+00}$                                     & $5.93\mathrm{E}{+00}$                                         & $7.62\mathrm{E}{+00}$                                                          \\
FMM                      & Charge deposition          & $1.44\mathrm{E}{-02}$                                    & $3.00\mathrm{E}{-02}$                                     & $8.86\mathrm{E}{-02}$                                         & $3.36\mathrm{E}{-01}$                                                          \\
                         & P2M                        & $4.20\mathrm{E}{-03}$                                    & $1.86\mathrm{E}{-02}$                                     & $7.65\mathrm{E}{-02}$                                         & $1.80\mathrm{E}{-01}$                                                          \\
                         & M2M                        & $6.42\mathrm{E}{-03}$                                    & $7.56\mathrm{E}{-03}$                                     & $8.22\mathrm{E}{-03}$                                         & $1.36\mathrm{E}{-01}$                                                          \\
                         & M2G                        & $3.75\mathrm{E}{-03}$                                    & $3.75\mathrm{E}{-03}$                                     & $3.88\mathrm{E}{-03}$                                         & $2.03\mathrm{E}{-02}$                                                          \\ \hline
Standard                 & Field interpolation        & $1.15\mathrm{E}{-02}$                                    & $1.18\mathrm{E}{-01}$                                     & $5.92\mathrm{E}{+00}$                                         & $7.60\mathrm{E}{+00}$                                                          \\
FMM                      & Field interpolation        & $2.83\mathrm{E}{-03}$                                    & $1.87\mathrm{E}{-02}$                                     & $9.46\mathrm{E}{-02}$                                         & $2.35\mathrm{E}{-01}$                                                          \\
                         & G2L + L2L                  & $9.57\mathrm{E}{-04}$                                    & $1.08\mathrm{E}{-03}$                                     & $1.22\mathrm{E}{-03}$                                         & $2.97\mathrm{E}{-02}$                                                          \\
                         & L2P                        & $1.75\mathrm{E}{-03}$                                    & $1.60\mathrm{E}{-02}$                                     & $9.32\mathrm{E}{-02}$                                         & $1.84\mathrm{E}{-01}$                                                          \\ \hline
Both             & Push particles             & $1.86\mathrm{E}{-03}$                                    & $2.02\mathrm{E}{-02}$                                     & $8.20\mathrm{E}{-02}$                                         & $9.09\mathrm{E}{-02}$                                                          \\
Both             & Field computation          & $1.10\mathrm{E}{-03}$                                    & $1.11\mathrm{E}{-03}$                                     & $1.12\mathrm{E}{-03}$                                         & $4.06\mathrm{E}{-03}$                                                         
\end{tabular}
\caption{SGCT-PIC method: wall-clock times for the different components of the FMM- and
standard-version algorithms, for $d=2$ and $p=1$.}
\label{tab:2}
\end{figure}

\begin{figure}[h!]
\centering
\begin{minipage}[]{0.49\textwidth}
 \includegraphics[width=\textwidth]{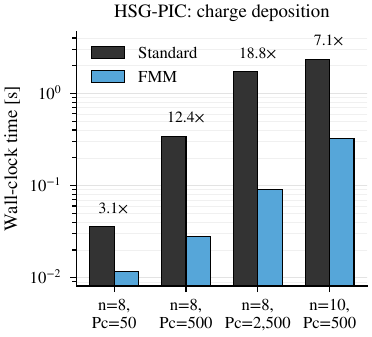}   
\end{minipage} 
 \begin{minipage}[]{0.49\textwidth}
 \includegraphics[width=\textwidth]{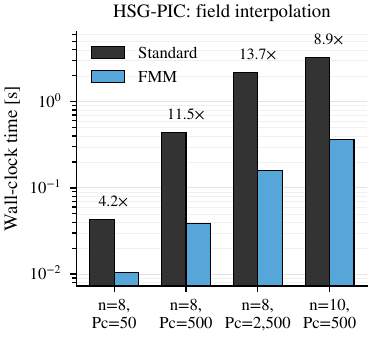}  
\end{minipage}  \\ \medskip
\begin{tabular}{llllll}
\hline
\multirow{2}{*}{Version} & \multirow{2}{*}{Operation} & \multicolumn{4}{c}{Time execution {[}s{]}}                                                                                                                                                                                                                            \\
                         &                            & \begin{tabular}[c]{@{}l@{}}$n=8$\\ $P_c=50$\end{tabular} & \begin{tabular}[c]{@{}l@{}}$n=8$\\ $P_c=500$\end{tabular} & \begin{tabular}[c]{@{}l@{}}$n=8$\\ $P_c=2{,}500$\end{tabular} & \multicolumn{1}{c}{\begin{tabular}[c]{@{}c@{}}$n=10$\\ $P_c=500$\end{tabular}} \\ \hline
Standard                 & Charge deposition          & $3.55\mathrm{E}{-02}$                                    & $3.43\mathrm{E}{-01}$                                     & $1.71\mathrm{E}{+00}$                                         & $2.33\mathrm{E}{+00}$                                                          \\
FMM                      & Charge deposition          & $1.16\mathrm{E}{-02}$                                    & $2.77\mathrm{E}{-02}$                                     & $9.11\mathrm{E}{-02}$                                         & $3.26\mathrm{E}{-01}$                                                          \\
                         & P2M                        & $4.11\mathrm{E}{-03}$                                    & $1.87\mathrm{E}{-02}$                                     & $8.09\mathrm{E}{-02}$                                         & $1.86\mathrm{E}{-01}$                                                          \\
                         & M2M                        & $6.14\mathrm{E}{-03}$                                    & $7.65\mathrm{E}{-03}$                                     & $8.64\mathrm{E}{-03}$                                         & $1.32\mathrm{E}{-01}$                                                          \\
                         & M2G                        & $1.31\mathrm{E}{-03}$                                    & $1.34\mathrm{E}{-03}$                                     & $1.43\mathrm{E}{-03}$                                         & $7.05\mathrm{E}{-03}$                                                          \\ \hline
Standard                 & Field interpolation        & $4.32\mathrm{E}{-02}$                                    & $4.39\mathrm{E}{-01}$                                     & $2.19\mathrm{E}{+00}$                                         & $3.23\mathrm{E}{+00}$                                                          \\
FMM                      & Field interpolation        & $1.04\mathrm{E}{-02}$                                    & $3.82\mathrm{E}{-02}$                                     & $1.60\mathrm{E}{-01}$                                         & $3.64\mathrm{E}{-01}$                                                          \\
                         & G2L + L2L                  & $5.41\mathrm{E}{-03}$                                    & $6.63\mathrm{E}{-03}$                                     & $7.77\mathrm{E}{-03}$                                         & $1.58\mathrm{E}{-01}$                                                          \\
                         & L2P                        & $4.55\mathrm{E}{-03}$                                    & $3.11\mathrm{E}{-02}$                                     & $1.52\mathrm{E}{-01}$                                         & $1.96\mathrm{E}{-01}$                                                          \\ \hline
Both             & Push particles             & $1.89\mathrm{E}{-03}$                                    & $1.93\mathrm{E}{-02}$                                     & $8.87\mathrm{E}{-02}$                                         & $9.83\mathrm{E}{-02}$                                                          \\
Both             & Field computation          & $9.39\mathrm{E}{-03}$                                    & $9.59\mathrm{E}{-03}$                                     & $9.96\mathrm{E}{-03}$                                         & $1.74\mathrm{E}{-01}$                                                         
\end{tabular}
\caption{HSG-PIC method: wall-clock times for the different components of the FMM- and
standard-version algorithms, for $d=2$ and $p=1$.}
\label{tab:4}
\end{figure}

For the SGCT-PIC method, the FMM version provides a substantial reduction in the cost of both charge deposition and field interpolation. The gain significantly increases with the number of particles. For $n=8$, charge deposition is accelerated by approximately $8.2\times$, $38.3\times$, and $66.9\times$ for $P_c=50$, $500$, and $2{,}500$, respectively. For $n=10$ and $P_c=500$, the speedup is $22.7\times$. Field interpolation achieves speedups of $4.1\times$, $6.3\times$, and $62.6\times$ for $n=8$ and $P_c=50$, $500$, and $2{,}500$, respectively, and $32.3\times$ for $n=10$, $P_c=500$. 

 The decomposition of the FMM timings shows that P2M and L2P increasingly dominate the particle-dependent part of the computation as $P_c$ increases, whereas M2M, M2G, and G2L+L2L remain comparatively inexpensive. In particular, the costs of M2G and G2L+L2L are nearly independent of $P_c$ for a fixed mesh resolution, as expected since these operations are primarily determined by the mesh hierarchy rather than by the number of particles. For $n=10$, their costs increase because of the larger number of boxes and levels. These results indicate that the FMM algorithm effectively replaces the particle-dependent operations of the particle--mesh coupling by operations whose cost scales much more favorably with the number of particles. 

The particle push and field computation remain a small fraction of the total cost and are identical in both implementations, so the overall performance is primarily determined by charge deposition and field interpolation. For $n=8$, the total cycle time is reduced by approximately $6.6\times$, $25.4\times$, and $44.7\times$ for $P_c=50$, $500$, and $2{,}500$, respectively. For $n=10$ and $P_c=500$, the speedup is $26.2\times$.

For HSG-PIC, the charge-deposition speedup increases from $3.1\times$ to $12.4\times$ and $18.8\times$ as $P_c$ increases from $50$ to $500$ and $2{,}500$, respectively, for $n=8$, and is $7.1\times$ for $n=10$, $P_c=500$. Field interpolation achieves speedups of $4.2\times$, $11.5\times$, and $13.7\times$ for $n=8$ and $P_c=50$, $500$, and $2{,}500$, respectively, and $8.9\times$ for $n=10$, $P_c=500$. As for SGCT-PIC, the P2M and L2P operations become increasingly important with increasing particle number, while the remaining FMM operations are comparatively insensitive to $P_c$. For $n=8$, the total iteration speedup increases from $2.7\times$ to $8.6\times$ and $11.4\times$ for $P_c=50$, $500$, and $2{,}500$, respectively, and is $8.5\times$ for $n=10$, $P_c=500$. Overall, the FMM algorithms substantially reduce the computational cost of the particle--mesh coupling in both SGCT- and HSG-PIC while preserving the standard results up to roundoff errors.

\section{Conclusion}
\label{sec:concl}
In this report, we have introduced two hierarchical algorithms for charge deposition and electric-field interpolation, which constitute the major computational bottleneck of the SGCT- and HSG-PIC methods. The method is based on a hierarchical directed acyclic graph (DAG) of particle-populated boxes that allows to exploit interactions between clusters of particles and mesh nodes through multipole and local expansions. These interactions are governed by piecewise polynomial kernels so that the associated multipole expansions are exact, requiring neither truncation nor approximation, and are globally valid throughout the domain, thereby eliminating the need for multipole-to-local translations. The number of arithmetic operations required to perform the two charge deposition and electric-field interpolation steps is reduced from $\O(p^dn^{d-1}N))$ to $\O(p^d(N+M))$, where $M=2^{dn}$ denotes the number of full-grid mesh nodes.  Numerical experiments in two dimensions have demonstrated significant speedups compared with standard implementations. Charge deposition is accelerated by factors of $8.2\times$--$66.9\times$ for SGCT-PIC and $3.1\times$--$18.8\times$ for HSG-PIC, while field interpolation achieves speedups of $4.1\times$--$62.6\times$ and $4.2\times$--$13.7\times$, respectively, depending on the particle-per-cell ratio. These speedups increase with the particle-per-cell ratio, as the reduced dependence of the hierarchical algorithms on the number of particles becomes increasingly advantageous for large particle populations.

Further work includes assessing the performance of the proposed algorithms in three-dimensional settings, where sparse-grid PIC methods can provide substantially larger speedups compared with standard PIC methods. We will also investigate parallel implementations that exploit the hierarchical structure of the algorithms and assess their scalability for large particle and mesh populations.

\section*{Acknowledgements}
This work has been supported by a grant from the French National Research Agency (ANR) project MATURATION (reference ANR-22-CE46-0012).
\bibliography{bib/bib}
\bibliographystyle{plain}
\appendix

\section{Appendix: diocotron instability test case}
\label{apd:1}

We consider a radially symmetric Gaussian ring as the initial spatial distribution of electrons, defined by
\begin{align*}
f^0_{\bm{x}}(\bm{x})=\gamma \exp\left(-\frac{(\|\bm{x}-\frac{L}{2}\|_2-\frac{L}{4})^2}{2(\beta L)^2} \right), \quad \text{with}~\gamma ~\text{such that} \int_{\Omega}f^0_{\bm{x}} dxdy =1,
\end{align*}
and $\beta = 0.03$, along with a Maxwellian velocity distribution given by
\begin{align*}
f^0_{\bm{v}}(\bm{v}) = \left(\frac{1}{\sqrt{\pi} v_T}\right)^3 \exp\Big(-\frac{\|\bm{v}\|_2^2}{v_T^2}\Big), \quad v_T = \sqrt{2 T_e q_e / m_e}.
\end{align*}
Here, $\|\cdot\|_2$ denotes the Euclidean norm. The small value of $\beta$ ensures a narrow Gaussian ring with strong spatial variations that are not aligned with the axes.
The domain size is set to $L=60$, and an external uniform magnetic field is applied along the $z$-axis, $\bm{B}_0 = (0,0,B_z)$, with $B_z=15$.  

The magnetic field induces an instability that deforms the initially radially symmetric electron density, forming vortices\cite{driscoll90,muralikrishnan21,deluzet22}. The strong magnetic field also imposes a constraint on the time step: it must be smaller than the electron gyroperiod, with cyclotron frequency $\Omega_c = B_z$. We set $\Delta t = 0.02$ to satisfy $\Omega_c \Delta t \leq 1$.
\end{document}